\documentclass[10pt]{amsart}
\usepackage{amssymb}
\usepackage{color}
\newtheorem{theorem}{Theorem}[section]

\newtheorem{lemma}[theorem]{Lemma}

\theoremstyle{definition}
\newtheorem{definition}[theorem]{Definition}
\newtheorem{remark}[theorem]{Remark}

\numberwithin{equation}{section}
\begin{document}

\title[Schwarz lemma for harmonic functions in the unit ball]{Schwarz lemma for harmonic functions in the unit ball}

\author[Z.  Xu]{Zhenghua Xu}
\author[T.  Yu]{Ting Yu}
\author[Q.  Huo]{Qinghai Huo}
%\thanks{This work was supported by the Natural Science Foundation of Anhui Province (No. 2308085MA04) and the   National Natural Science Foundation of China  (No. 11801125). }
\address{Zhenghua Xu, School of Mathematics, Hefei University of Technology, Hefei 230601, China}
\email{zhxu$\symbol{64}$hfut.edu.cn}

\address{Ting Yu, School of Mathematics, Hefei University of Technology, Hefei 230601, China}
\email{2022111460$\symbol{64}$mail.hfut.edu.cn}

\address{Qinghai Huo, School of Mathematics, Hefei University of Technology, Hefei 230601, China}
\email{hqh86$\symbol{64}$mail.ustc.edu.cn}

\begin{abstract}
Recently, it is proven that positive harmonic functions defined in the   unit disc or the upper half-plane in $\mathbb{C}$   are contractions in hyperbolic metrics \cite{Markovic}. Furthermore,  the same result does not hold in higher dimensions  as   shown by given counterexamples \cite{Melentijevic-P}. In this paper, we shall show that positive (or bounded) harmonic functions defined in the unit ball in $\mathbb{R}^{n}$  are  Lipschitz  in   hyperbolic metrics.
The involved method   in main results    allows  to  establish essential improvements of  Schwarz type inequalities  for monogenic functions in Clifford analysis \cite{Zhang14,Zhang16} and octonionic analysis  \cite{Wang-Bian-Liu} in a unified  approach.
\end{abstract}
\maketitle
\section{Introduction}

Let $\mathbb{B}_{n}$ be the open unit ball and  $\mathrm{\mathbb{H}}_{n}$  be the upper half-space in $\mathbb{R}^{n}$, respectively.  Specially,  $\mathbb{B}_{2}$ and $\mathrm{\mathbb{H}}_{2}$ is denoted as $\mathbb{D}$ and $\mathbb{H}$, identified with the open unit disc  and the upper half-plane of $\mathbb{C}$.  %the upper half-plane $\mathbb{H}=\{ z\in \mathbb{C}:   {\rm Im} z>0 \}$
 The classical Schwarz-Pick lemma states that   holomorphic functions  $f: \mathbb{D}\rightarrow \mathbb{D}$ satisfy
\begin{eqnarray}  \label{Schwarz-Pick-der}
|f'(z)|\leq \frac{1-|f(z)| ^{2}}{1-|z| ^{2}}, \quad  z\in \mathbb{D},
\end{eqnarray}
and\begin{eqnarray}  \label{Schwarz-Pick}
|\varphi_{f(w)}(f(z))|\leq |\varphi_{w}(z)|, \quad z,w\in \mathbb{D},
\end{eqnarray}
where $\varphi_{w}(z)=(w-z)(1-\overline{w}z)^{-1}$ is the M\"{o}bius transformation  of $\mathbb{D}$ onto itself.

Recall that  the hyperbolic metric on $\mathbb{D}$ is given by
$$ d_{\mathbb{D}}(z,w)=\log \frac{1+|\varphi_{w}(z)|}{1-|\varphi_{w}(z)|}=2\tanh^{-1}(|\varphi_{w}(z)|), \quad z,w\in \mathbb{D}.$$
 Note that $\tanh^{-1}$ is monotone increasing, then (\ref{Schwarz-Pick}) can be rewritten as
$$d_{\mathbb{D}}(f(z),f(w))\leq d_{\mathbb{D}}(z,w),  \quad z,w\in \mathbb{D}. $$
That is to say every holomorphic function $f:\mathbb{D}\rightarrow \mathbb{D}$ is a contraction with respect to the hyperbolic metric on   $\mathbb{D}$.

In 2012, Kalaj and Vuorinen in \cite[Theorem 1.12]{Kalaj-Vuorinen} proved that, for harmonic functions $f:\mathbb{D} \rightarrow (-1,1)$,
\begin{eqnarray}  \label{Kalaj-Vuorinen}
|\triangledown f (z)|\leq \frac{4}{\pi} \frac{1-|f(z)| ^{2}}{1-|z| ^{2}}, \quad  z\in \mathbb{D},
\end{eqnarray}
where the constant $4/\pi$ is sharp.

Equivalently,  harmonic functions $f:\mathbb{D} \rightarrow (-1,1)$  are  Lipschitz  in the hyperbolic metric, i.e.,
$$d_{\mathbb{D}}(f(z),f(w))\leq \frac{4}{\pi} d_{\mathbb{D}}(z,w), \quad z,w\in \mathbb{D},$$
which holds also for   harmonic functions  defined in hyperbolic plane domains (see \cite[Theorem 4]{Melentijevic-M}).

In 2013, Chen in \cite[Theorem 1.2]{Chen} obtained a sharper  version of (\ref{Kalaj-Vuorinen})
$$|\triangledown f (z)|\leq \frac{4}{\pi} \frac{\cos \frac{\pi|f(z)|}{2} }{1-|z| ^{2}}, \quad  z\in \mathbb{D},$$
which was generalized into pluriharmonic functions (see \cite[Theorem 1.5]{Xu}).
%and \cite[Theorem 2.3]{Kalaj}

Motivated by the result of Kalaj and Vuorinen,  Markovi\'{c} in 2015 showed \cite[Theorem 1.1]{Markovic} that harmonic functions $f:\mathbb{H} \rightarrow \mathbb{R}^{+}=(0,+\infty)$ are  contractible   in the hyperbolic metric, i.e.,
\begin{eqnarray}  \label{Markovic}
d_{\mathbb{R}^{+}}(f(z),f(w))\leq   d_{\mathbb{H}}(z,w), \quad z,w\in \mathbb{H},\end{eqnarray}
where  the hyperbolic  metric $d_{\mathbb{H}}$ on the upper half-plane $\mathbb{H}$ is given by
$$d_{\mathbb{H}}(z,w)=2\tanh^{-1}\Big| \frac{z-w}{\overline{z}-w}\Big|, \quad z,w\in \mathbb{H}.$$
In particular, the hyperbolic distance $d_{\mathbb{R}^{+}}$ on $\mathbb{R}^{+}$ is
$$ d_{\mathbb{R}^{+}}(x,y)=d_{\mathbb{H}}(ix,iy)=|\log\frac{x}{y}|, \quad x,y\in \mathbb{R}^{+}.$$

In \cite{Melentijevic-P},  Melentijevi\'{c} established  some refinements of Schwarz's lemma for holomorphic functions with the invariant gradient and gave another proof of (\ref{Markovic}) based on Harnack inequality. By using the same strategy, one can show that  harmonic functions $f:\mathbb{D} \rightarrow \mathbb{R}^{+}$ are also contractible   in the hyperbolic metric,
\begin{eqnarray}  \label{Markovic-disc}d_{\mathbb{R}^{+}}(f(z),f(w))\leq   d_{\mathbb{D}}(z,w), \quad z,w\in \mathbb{D}.\end{eqnarray}
 Furthermore, Melentijevi\'{c}  provide counterexamples to show that these results  in (\ref{Markovic}) and (\ref{Markovic-disc}) do not hold in higher dimensions for positive harmonic functions defined in $\mathbb{B}_n$ or $\mathbb{H}_n$ when $n\geq3$; see \cite[Example 1 and Example 2]{Melentijevic-P}.

In fact, up to multiplying a constant depending on the dimension,  these results  in (\ref{Markovic}) and (\ref{Markovic-disc}) still hold for  positive harmonic functions defined in higher dimensions. To be more precise, we shall establish the following result in this paper.
\begin{theorem} \label{Main-Ball}
Let $n\geq2$ be integer and   $f:\mathbb{B}_{n} \rightarrow \mathbb{R}^{+}$   be a harmonic function. Then
\begin{eqnarray}  \label{Ball}
d_{\mathbb{R}^{+}}(f(x),f(y))\leq  (n-1) d_{\mathbb{B}_{n}}(x,y), \quad x,y\in \mathbb{B}_{n},
\end{eqnarray}
where $d_{\mathbb{B}_{n}}$ is the hyperbolic metric on $\mathbb{B}_{n}$   given by $d_{\mathbb{B}_{n}}(x,y) =2\tanh^{-1}(|\varphi_{y}(x)|)$, and $\varphi_{y}(x)$ is the M\"{o}bius transformation of $\mathbb{B}_{n}$ defined  by (\ref{Mobius-transformation}).
\end{theorem}

The proof of Theorem \ref{Main-Ball} is built on the following estimate. Moreover, this estimate   is sharp.
\begin{theorem}\label{Main-sharp-theorem}
Let $n\geq2$ be integer and  $f:\mathbb{B}_{n} \rightarrow \mathbb{R}^{+}$   be a harmonic function. Then
 \begin{eqnarray} \label{Main-sharp}
|(|x|^{2}-1)\nabla f(x)+(n-2)xf(x)|\leq n f(x), \quad x\in \mathbb{B}_{n}.
\end{eqnarray}
If the equality in  (\ref{Main-sharp}) is attained  for some  $a \in \mathbb{B}_{n}$, then    there is  a point $\xi \in   \partial \mathbb{B}_{n}$ such that
  \begin{eqnarray} \label{extremal-function}
  f(x)=f(a)|1-\varphi_{a}(x)\overline{a}|^{n-2}P_{\xi}\circ \varphi_{a}(x), \quad x\in \mathbb{B}_{n},  \end{eqnarray}
where  $P_{\xi}$ is the Poisson kernel given by
 $$P_{\xi}(x)=P(x, \xi)=\frac{1-|x|^{2}}{|x-\xi|^{n}}.$$
 Moreover, every positive and harmonic function $f$ defined by (\ref{extremal-function}) satisfies  the equality in  (\ref{Main-sharp})  for all  $x \in \mathbb{B}_{n}$ and $\xi \in   \partial \mathbb{B}_{n}$.

\end{theorem}

The natural question is to ask: what is the analogue  of Theorem  \ref{Main-Ball} for bounded harmonic functions $f:\mathbb{B}_{n} \rightarrow (-1,1)$?  Based on the  proved Khavinson conjecture in \cite{Liu}, Liu  very recently has given  an answer to this question by established the following  Schwarz-Pick type  inequality \cite[Theorem 1]{Liu2}, which can be viewed as a counterpart of  Theorem \ref{Main-sharp-theorem} for bounded harmonic functions.

\begin{theorem}\label{theorem-Liu}
Let $n\geq4$ be integer and   $f:\mathbb{B}_{n} \rightarrow (-1,1)$   be a harmonic function. Then
\begin{eqnarray}  \label{Liu}
|\nabla  f(x)|\leq \frac{|\mathbb{B}_{n-1}|}{|\mathbb{B}_{n}|}\frac{2}{1-|x|^{2}},\quad x\in \mathbb{B}_{n},
\end{eqnarray}
where $|\mathbb{B}_{n}|$  denotes the  volume of the unit ball $\mathbb{B}_{n}$. The equality in (\ref{Liu}) holds if and only
if $x=0$ and $f=U \circ T$ for some  $T\in O(n)$, where $U$ is the Poisson integral of the function that equals $1$ on a hemisphere and $-1$ on the remaining hemisphere and  $O(n)$ denotes the set of orthogonal transformations of $\mathbb{R}^{n}$.
\end{theorem}

Note that, for $n=2$,  (\ref{Liu}) can be obtained directly from (\ref{Kalaj-Vuorinen}). Curiously, for $n=3$,  (\ref{Liu}) should be replaced by
 $$|\nabla  f(x)|< \frac{8}{3\sqrt{3}}\frac{1}{1-|x|^{2}},\quad x\in \mathbb{B}_{3},$$
where $\frac{8}{3\sqrt{3}} (>2\frac{|\mathbb{B}_{2}|}{|\mathbb{B}_{3}|}=1.5)$ is the best possible; see \cite[Note]{Khavinson} and \cite[Remark 1]{Liu}.

\begin{remark}
Factually, (\ref{Liu}) at $x=0$ holds for all $n\geq2$ and the constant $2 |\mathbb{B}_{n-1}|/|\mathbb{B}_{n}|$  is optimal in this case; see \cite[Theorem 6.26]{Axler} or \cite[Corollary 2.2]{Kalaj}. Furthermore, the requirement that $f$ is real-valued   is crucial in the validity of (\ref{Liu}). In fact, (\ref{Liu}) fails even at $x=0$ for complex-valued harmonic functions \cite[p. 126]{Axler}.
In this paper, for  vector-valued harmonic functions $f:\mathbb{B}_{n}\rightarrow \mathbb{R}^{m}$, we find that (\ref{Liu}) still hold by using the matrix (operator) norm of the Jacobian matrix  $\nabla  f(x)\in \mathbb{R}^{m\times n}$,  that is the square root of the biggest eigenvalue of $(\nabla  f(x))^{T} \nabla  f(x)$.
\begin{theorem} \label{Theorem-Liu-vector}
Let  $f:\mathbb{B}_{n}\rightarrow \mathbb{B}_{m}$ be harmonic functions with $n=2,$ or $n\geq4$. Then
\begin{eqnarray}  \label{Liu-vector}
\|\nabla  f(x)\|\leq \frac{|\mathbb{B}_{n-1}|}{|\mathbb{B}_{n}|}\frac{2}{1-|x|^{2}},\quad x\in \mathbb{B}_{n},
\end{eqnarray}
where $\|\nabla  f(x)\|$  denotes the matrix norm of $\nabla  f(x)\in \mathbb{R}^{m\times n}$.
\end{theorem}

We  restate (\ref{Liu-vector}) in the terms of the hyperbolic metric as follows. The proof is standard and omitted here.

\begin{theorem} \label{Liu-Hyperbolic}
Let  $f:\mathbb{B}_{n}\rightarrow \mathbb{B}_{m}$ be harmonic functions with $n=2,$ or $n\geq4$. Then
\begin{eqnarray}  \label{Liu-Hyperbolic}
|f(x)-f(y)| \leq \frac{|\mathbb{B}_{n-1}|}{|\mathbb{B}_{n}|}   d_{\mathbb{B}_{n}}(x,y), \quad x,y\in \mathbb{B}_{n}.
\end{eqnarray}
\end{theorem}

\begin{remark}The distance in the left side of (\ref{Liu-Hyperbolic}) is Euclidean but not hyperbolic. Based on the inequality (\ref{Kalaj-Vuorinen}) and Theorem  \ref{Main-Ball}, one would conjecture a sharper version of (\ref{Liu}) that, for  harmonic functions $f:\mathbb{B}_{n}\rightarrow(-1,1)$ with $n\geq4$,
$$ \frac{|\nabla  f(x)|}{1-|f(x)|^{2}}\leq  \frac{|\mathbb{B}_{n-1}|}{|\mathbb{B}_{n}|}\frac{2}{1-|x|^{2}}, \quad x\in \mathbb{B}_{n}.$$
However, it is not the case as shown by a counter-example \cite[Theorem 2.1]{Khalfallah}.
\end{remark}

\begin{remark}Let  $f:\mathbb{B}_{n}\rightarrow \mathbb{R}^{m}$. When $m=1$, the  matrix norm concises with the  Euclidean norm of $\nabla  f(x)\in \mathbb{R}^{n}$, i.e., $\|\nabla  f(x) \|=|\nabla  f(x)|$.
 Furthermore, it holds that
 $$ |\triangledown |f| (x)| \leq  \|\nabla  f(x) \|, \quad x\in \mathbb{B}_{n},$$
 where $\triangledown |f| (x) =(\frac{\partial |f|}{\partial x_{1}},\frac{\partial |f|}{\partial x_{2}},\ldots\frac{\partial |f|}{\partial x_{n}})$ denotes the gradient of the  Euclidean norm  of $f(x)$.
\end{remark}

In the study of the  Schwarz-Pick inequality for  holomorphic functions, the quantity $|\triangledown |f||$ was first adopted by Pavlovi\'{c}  \cite{Pavlovic} due to that the  classical form in (\ref{Schwarz-Pick-der}) does not hold generally for vector-valued holomorphic functions. To obtain analogous form of (\ref{Schwarz-Pick-der}), Pavlovi\'{c} gave that, for holomorphic mappings $f=(f_{1}, \ldots, f_{n}): \mathbb{D} \rightarrow \mathbb{C}^{n}$ with  $|f|=(|f_{1}|^{2}+\cdots +|f_{n}|^{2})^{1/2}<1$,
\begin{equation} \label{Schwarz-pick-norm}
 |\triangledown |f| (z)|\leq \frac{1-|f(z)| ^{2}}{1-|z| ^{2}}, \quad z\in \mathbb{D}.
\end{equation}

Following the idea of Pavlovi\'{c},  Chen and Hamada established the  vector-valued version of the Khavinson conjecture for the norm of harmonic functions from  the Euclidean unit ball $\mathbb{B}_{n}$ into the unit ball of the real Minkowski space
by complicated calculations \cite{Chen-Hamada}.  By the same technique,  they  gave several sharp Schwarz-Pick type inequalities for pluriharmonic functions from the Euclidean unit ball (or  the unit polydisc) in   $\mathbb{C}^{n}$ into the unit ball of the Minkowski space.
Very recently,  using the technique of the present author  \cite{Xu2},  Chen et al.   have provided some improvements and generalizations of the corresponding results in  \cite{Chen-Hamada} into Banach spaces by  a relatively simple  proof \cite{Chen-Hamada-Ponnusamy}.
\end{remark}

As a large subclass of the  harmonic functions, the concept of monogenic functions appears in Clifford analysis, which is also a natural generalization of complex analysis into higher dimensions over non-commutative algebras.  For monogenic functions,   the Schwarz lemma does not hold at least in the original form, observed by Yang and Qian \cite[Remark 2]{Yang-Qian} and they established  a Schwarz lemma outside of the unit
ball in $\mathbb{R}^{n+1}$. Recently, some analogues of Schwarz  lemma  inside the unit ball were obtained in Clifford analysis \cite{Zhang14,Zhang16}, quaternionic analysis \cite{Gu,Wang-Sun-Bian} and octonionic analysis \cite{Wang-Bian-Liu}.  For example, by integral representations of harmonic functions  and M\"{o}bius transformations with Clifford coefficients,  Zhang established the following  Schwarz type lemma.

\begin{theorem}$($\cite[Theorem 3.2]{Zhang16}$)$ \label{Zhang}
Let  $f:\mathbb{B}_{n+1} \rightarrow \mathbb{R}_{0,n}$   be a Clifford algebra valued monogenic function with $|f(x)|\leq1,   x \in \mathbb{B}_{n+1}$. If $f(a)=0$ for some $a\in \mathbb{B}_{n+1}$, then
\begin{eqnarray}  \label{Zhang}
|f(x)| \leq \frac{(1+|a|)^{n}}{ \sqrt[n+1]{2}-1} \frac{|x-a|}{|1-\overline{a}x|^{n+1}}, \quad x \in \mathbb{B}_{n+1},
\end{eqnarray}
where $|\cdot |$ is the norm   and $\overline{ \cdot } $ is the  conjugate in $\mathbb{R}_{0,n}$.
\end{theorem}

By the same  technique as in the prove of  Theorem  \ref{Main-Ball},   we shall offer  a unify method to establish Schwarz type inequalities for harmonic functions in Clifford analysis and octonionic analysis as follows, instead of monogenic functions.

\begin{theorem}\label{Theorem-Zhang-improved}
Let  $f:\mathbb{B}_{n+1} \rightarrow \mathbb{R}_{0,n}$   be a  Clifford algebra valued  harmonic function with $|f(x)|\leq1,  x \in \mathbb{B}_{n+1}$. If $f(a)=0$ for some $a\in \mathbb{B}_{n+1}$, then
\begin{eqnarray}  \label{Zhang-improved}
|f(x)| \leq \frac{(1+|a|)^{n-1}}{ \sqrt[n+1]{2}-1} \frac{|x-a|}{|1-\overline{a}x|^{n}}, \quad x \in \mathbb{B}_{n+1},
\end{eqnarray}
and as a corollary
$$|\nabla f(a)| \leq \frac{1}{ \sqrt[n+1]{2}-1} \frac{1}{(1+|a|)(1- |a|)^{n}}.$$
\end{theorem}

\begin{theorem}\label{Wang}
Let  $f:\mathbb{B}_{8} \rightarrow \mathbb{O}$   be  an octonion valued  harmonic function with   $|f(x)|\leq1,  x \in \mathbb{B}_{8}$. If $f(a)=0$ for some $a\in \mathbb{B}_{8}$, then
$$|f(x)| \leq \frac{(1+|a|)^{6}}{ \sqrt[8]{2}-1} \frac{|x-a|}{|1-\overline{a}x|^{7}}, \quad x \in \mathbb{B}_{8}.$$
where $|\cdot |$ is the norm  and $\overline{ \cdot } $ is the  conjugate in  $\mathbb{O}$.
\end{theorem}

 Note that  $|1-\overline{a}x|<1+|a|$ for $a,x \in \mathbb{B}_{n+1}$.  Hence,   in a broader function class being harmonic, the obtained results in Theorems \ref{Theorem-Zhang-improved} and \ref{Wang} are   essential improvements of  monogenic versions in \cite[Theorem 4.8]{Zhang14}, \cite[Theorem 3.2]{Zhang16} and \cite[Theorem 4]{Wang-Bian-Liu}, respectively.

The remaining part of the paper is organized as follows. The next section  shall recall preliminaries on Clifford algebras and  use it to rewrite  some known properties of M\"{o}bius transformations of $\mathbb{B}_{n}$, which shall be used in the proof of main results. The section $3$ is devoted to the proof of Theorems \ref{Main-Ball}, \ref{Main-sharp-theorem}  and \ref{Theorem-Liu-vector}. In section $4$, we recall the concepts of monogenic functions  in Clifford analysis and octonionic analysis  and show that they are subclasses of harmonic functions. Finally, we give the proof of   Theorems \ref{Theorem-Zhang-improved} and \ref{Wang}.

\section{Preliminaries}
In this section, we first  recall preliminaries on Clifford algebras; see e.g. \cite{Gurlebeck}.

Denote by $\mathbb{R}_{0,n}$ the real Clifford algebra  over imaginary units  $\{e_1,e_2,\ldots,e_n\}$ which satisfy
$$e_ie_j+e_je_i=-2\delta_{ij}e_{0}, \quad  1\leq i\leq j\leq n,$$
where $e_{0}$  is identify with $1$, $\delta_{ij}$ is Kronecker function.

Each element $a\in \mathbb{R}_{0,n}$  has the form of
$$a=\sum_{A}a_{A}e_{A},\quad a_A \in \mathbb{R},$$
where  $A = h_{1}h_{2} \ldots h_{r}$ with $ 1\leq h_{1}<h_{2}<\ldots<h_{r}\leq n, e_{A}=e_{h_{1}}e_{h_{2}}\ldots e_{h_{r}}$ and $e_{\emptyset}=e_{0}=1$.   The \textit{real} part of $a\in \mathbb{R}_{0,n}$ is ${\rm{Re}}\, a=a_{\emptyset}=a_0$.
The norm of  $a$ is defined by $|a|= ({\sum_{A}|a_{A}|^{2}} )^{\frac{1}{2}}.$ As a real  vector space, the dimension of Clifford algebra $\mathbb{R}_{0,n}$  is $2^{n}$. The paravector $x$ in $\mathbb{R}_{0,n}$ is given by $$x = \sum _{i=0}^{n} x_{i}e_{i},  \quad x_{i}\in \mathbb{R}.$$ Hence, the space $\mathbb R^{n+1}$ can be identified as the set of all  paravector  in Clifford algebra $\mathbb{R}_{0,n}$.
For paravector $x\neq0$, it  inverse is given by
$$x^{-1}=  \frac{\overline{x}}{|x|^{2}}, $$
where $\overline{x}$ denotes the  conjugate of $x$, that is $\overline{x}= \sum _{i=0}^{n} x_{i}\overline{e_{i}}
= x_{0}-\sum _{i=1}^{n} x_{i}{e_{i}}$.
 Note that Clifford algebra is   associative and non-commutative, but not divisible generally.  The equality $|ab|=|a| |b|$ does not  hold generally for   $a, b \in \mathbb{R}_{0,n}$ when  $n\geq 3$. However,  it is  holds for in the following special case in  Clifford algebras (see \cite[Theorem 3.14]{Gurlebeck}).
\begin{lemma}\label{modulus}
Let $a\in \mathbb{R}_{0,n} $ and $x\in \mathbb R^{n+1}$. Then
$$|ax|=|xa|=|a||x|.$$
 \end{lemma}

Now we introduce some known properties of M\"{o}bius transformations of $\mathbb{B}_{n}$    by using the language of Clifford algebras, which shall be used in the sequeal. These results can be founded in \cite{Ahlfors,Stoll}.

It is known that any M\"{o}bius transformation $\psi$ of  $\mathbb{B}_{n}$ onto  itself has the form  $\psi=T \varphi_{a}$, where $T\in O(n)$ and $\varphi_{a}$ is M\"{o}bius transformations of  $\mathbb{B}_{n}$ with $\varphi_{a}(0)=a\in \mathbb{B}_{n}$, given by
\begin{eqnarray} \label{Mobius-transformation}
\varphi_{a}(x)=\frac{(1-|a|^{2})(a-x)+|a-x|^{2}a}{[x,a]^{2}},\quad x\in \mathbb{B}_{n}, \end{eqnarray}
where $$ [x,a]=\sqrt{1+|a|^{2}|x|^{2}-2\langle a, x\rangle}.$$
Here $\langle \cdot, \cdot\rangle$ is real inner product in $\mathbb{R}^{n}$.

The vector space $\mathbb{R}^{n}$ can be viewed as the paravector  in   $\mathbb{R}_{0,n-1}$,  and then
 $$\langle a, x\rangle={\rm{Re}}\,(x\overline{a})={\rm{Re}}\,(\overline{a}x).$$
Hence, $$[x,a]=\sqrt{1+|a|^{2}|x|^{2}-2{\rm{Re}}\,(x\overline{a})}=|1-x\overline{a}|.$$
Consequently,  the mapping $\varphi_{a}$ can be expressed as
$$\varphi_{a}(x)=\frac{(1-|a|^{2}+a\overline{(a-x)})(a-x)}{|1-x\overline{a}|^{2}}=\frac{(1 -a\overline{x})(a-x)}{|1-x\overline{a}|^{2}}
=(1 -x\overline{a})^{-1}(a-x).$$
Furthermore, it holds that
 \begin{eqnarray} \label{inverse}
\varphi_{a}^{-1}=\varphi_{a}.\end{eqnarray}
To see this, we should notice  that
 $$\varphi_{a}(x)=(1 -x\overline{a})^{-1}(a-x)=(a-x)(1-\overline{a}x)^{-1}.$$
Let $y=\varphi_{a}(x)$. Then
 $$ a-x=(1-x\overline{a})y=y-x\overline{a}y,$$
  which implies that
  $$ a-y=x(1-\overline{a}y) \Rightarrow x=(a-y)(1-\overline{a}y)^{-1}=\varphi_{a}(y).$$

Denote by $\mathcal{M}(\mathbb{B}_{n})$ M\"{o}bius transformations of $\mathbb{B}_{n}$ onto $\mathbb{B}_{n}$.
Recall a useful identity \cite[Theorem 2.1.3]{Stoll}
$$\frac{|\varphi_(x)-\varphi_(y)|^{2}}{(1-|\varphi_(x)|^{2})(1-|\varphi_(y)|^{2})}=\frac{|x-y|^{2}}{(1-|x|^{2})(1-|y|^{2})}, \quad \varphi\in \mathcal{M}(\mathbb{B}_{n}).$$
This formula implies that
\begin{eqnarray} \label{gradient-x-y}
 \|\nabla \varphi_{a} (x)\|= \overline{\lim}_{y\rightarrow x} \frac{|\varphi_{a}(x)-\varphi_{a}(y)|}{|x-y|}=  \frac{1-|\varphi_{a} (x)|^{2}}{1-|x|^{2}}.
 \end{eqnarray}

From  the identity \begin{eqnarray} \label{1}
 1-|\varphi_{a}(x)|^{2}=1-|a-x|^{2}|1-x\overline{a}|^{-2}=\frac{(1-|a|^{2})(1-|x|^{2})}{|1-x\overline{a}|^{2}},
\end{eqnarray}
it is easy to see that
 \begin{eqnarray} \label{gradient}
\|\nabla \varphi_{a} (x)\|=\frac{1-|a|^{2}}{|1-x\overline{a}|^{2}}\in (\frac{1-|a|}{1+|a|},\frac{1+|a|}{1-|a|}).
 \end{eqnarray}

\section{Proof of   Theorems \ref{Main-Ball}, \ref{Main-sharp-theorem}  and \ref{Theorem-Liu-vector}}
To prove  main results, we recall   the invariance of the Laplace equation \cite[Chapter 1.10, p. 22]{Hua}.
\begin{lemma}\label{Hua}
Let $a\in \mathbb{B}_n$.  If an independent variable undergoes the transformation $y=\varphi_{a}(x), x\in \mathbb{B}_n,$ and the function is transformed by
 \begin{eqnarray} \label{transformation}
Y(y)=\Big(\frac{|1-x\overline{a}|^{2}}{1-|a|^{2}}\Big)^{\frac{n}{2}-1} X(x),
\end{eqnarray}
 then
$$ (1-|x|^{2})^{\frac{n}{2}+1}\sum _{i=0}^{n-1} \frac{\partial^{2} X}{\partial x_{i}^{2}}=(1-|y|^{2})^{\frac{n}{2}+1}\sum _{i=0}^{n-1} \frac{\partial^{2} Y}{\partial y_{i}^{2}}.$$
 \end{lemma}
Let $y=\varphi_{a}(x)$.  From the identity
 \begin{eqnarray} \label{x-y}
 |1-x\overline{a}||1-y\overline{a}|=1-|a|^{2},
 \end{eqnarray}
and  (\ref{inverse}),   the transformation in  (\ref{transformation}) can be rewritten as
  $$Y   (y)=\Big(\frac{1-|a|^{2}}{|1-y\overline{a}|^{2}}\Big)^{\frac{n}{2}-1} X( \varphi_{a}^{-1} (y) )=\Big(\frac{1-|a|^{2}}{|1-y\overline{a}|^{2}}\Big)^{\frac{n}{2}-1} X( \varphi_{a} (y) ).$$

Hence,   Lemma \ref{Hua} gives directly the following result.
  \begin{lemma}\label{Hua-0}
Let $a\in \mathbb{B}_n$ and $f$ be a harmonic function in $\mathbb{B}_n$. Then the function
$$f[\varphi_{a}](x):=\Big(\frac{1-|a|^{2}}{|1-x\overline{a}|^{2}}\Big)^{\frac{n}{2}-1} f\circ \varphi_{a}(x)$$
is still harmonic in $\mathbb{B}_{n}$.
 \end{lemma}

 Now we are in a position to prove Theorem  \ref{Main-sharp-theorem}.
\begin{proof}[Proof of  Theorem  \ref{Main-sharp-theorem}]
Firstly, let us prove the estimate (\ref{Main-sharp}) in the special case  $x=0$, i.e.
\begin{eqnarray} \label{Main-0}
|\nabla f(0)|\leq n f(0),
\end{eqnarray}
for the positive harmonic function $f$ defined in $\mathbb{B}_{n}$.

From the Poisson-Herglotz representation, it holds that
  $$f(x)=\int_{\partial \mathbb{B}_{n}}P(x, \xi)  d\mu(\xi), \quad x\in \mathbb{B}_{n},$$
where $\mu$ is a positive Borel measure on $\partial \mathbb{B}_{n}$ such that $\int_{\partial \mathbb{B}_{n}}  d\mu(\xi)=f(0)$.\\
Note that
\begin{eqnarray} \label{Poisson}
\nabla P_{\xi}(x)=-\frac{2x}{|x-\xi|^{n}}-\frac{n(1-|x|^{2})}{|x-\xi|^{n+2}}(x-\xi),  \quad (x, \xi)\in \mathbb{B}_{n} \times \partial \mathbb{B}_{n}.  \end{eqnarray}
In particular,
$$\nabla P_{\xi}(0)=n \xi,  \quad \xi \in \partial \mathbb{B}_{n}. $$
Hence, $$|\nabla f(0)|=n|\int_{\partial \mathbb{B}_{n}}  \xi d\mu(\xi)|\leq n\int_{\partial \mathbb{B}_{n}}    d\mu(\xi)=nf(0),$$
where the equality is attained if and only if the measure $\mu$ is  a singleton, that is to say, there exits some $\xi \in \partial \mathbb{B}_{n}$ such that $$\mu(\{\xi\})=f(0),\ \mu(\partial \mathbb{B}_{n}\setminus \{\xi\})=0.$$
This shows   (\ref{Main-0})   and that the equality in (\ref{Main-0}) is attained   if and only if  $f(x)=f(0)P_{\xi}(x)$ for some $\xi \in \partial \mathbb{B}_{n}$,  and in this case, we have, by (\ref{Poisson}),
 \begin{eqnarray*}
  & &  (|x|^{2}-1)\nabla f(x)+(n-2)xf(x)
\\
&=&  f(0)(|x|^{2}-1)\Big( -\frac{2x}{|x-\xi|^{n}}-\frac{n(1-|x|^{2})}{|x-\xi|^{n+2}}(x-\xi) -\frac{(n-2)x}{|x-\xi|^{n}} \Big)
\\
&=&f(0) \frac{n(|x|^{2}-1)}{|x-\xi|^{n+2}} ((|x|^{2}-1)(x-\xi)-x |x-\xi|^{2})
\\
&=& f(0)  \frac{n(|x|^{2}-1)}{|x-\xi|^{n+2}} ( |x|^{2}-1  -x \overline{(x-\xi)})(x-\xi)
\\
&=& f(0)  \frac{n(|x|^{2}-1)}{|x-\xi|^{n+2}} ( x \overline{\xi}-1)(x-\xi).
\end{eqnarray*}
By  Lemma \ref{modulus}, it follows that
$$|( x \overline{\xi}-1)(x-\xi)|=|  x \overline{\xi}-1||x-\xi|=|x-\xi|^{2}. $$
Therefore,  the harmonic function  $f(x)=f(0)P_{\xi}(x)$ satisfies the following identity
  \begin{eqnarray} \label{extremal-function-0}
   |(|x|^{2}-1)\nabla f(x)+(n-2)xf(x)|=nf(0)\frac{1-|x|^{2}}{|x-\xi|^{n}} =nf(x), \quad (x, \xi)\in \mathbb{B}_{n} \times \partial \mathbb{B}_{n}.
  \end{eqnarray}

Secondly, we prove the conclusion in the general case  $x=a$.  Fix $a\in \mathbb{B}_{n}$. By Lemma \ref{Hua-0},  the   function $f[\varphi_{a}](x)$ is harmonic in $\mathbb{B}_{n}$.   Hence,  by applying   the inequality (\ref{Main-0}) to the  positive and harmonic  function $f[\varphi_{a}](x)$, we have
\begin{eqnarray} \label{Main-0-h}
|\nabla f [\varphi_{a}] (0)| \leq n f[\varphi_{a}] (0).
\end{eqnarray}
Direct calculations  gives that
$$\nabla \frac{1}{|1-x\overline{a}|^{n-2}}=\frac{(2-n)(|a|^{2}x-a)}{|1-x\overline{a}|^{n-2}}, \quad x\in \mathbb{B}_{n},$$
and
$$ \nabla  (\frac{1}{|1-x\overline{a}|^{n-2}} f\circ \varphi_{a}(x))|_{x=0}= (|a|^{2}-1)\nabla f(a)+(n-2)af(a),$$
then (\ref{Main-0-h}) reduces into
$$|(|a|^{2}-1)\nabla f(a)+(n-2)af(a)|\leq n f(a). $$

Let  us consider the case of   the equality in (\ref{Main-sharp}) is attained at $x=a$, that is
$$|\nabla f [\varphi_{a}] (0)| =n f[\varphi_{a}] (0).$$
Therefore, the  previously obtained result of $x=0$ gives that $f[\varphi_{a}](x)=f[\varphi_{a}](0)P_{\xi}(x)$ for some $\xi \in \partial \mathbb{B}_{n}$. More precisely,
$$\Big(\frac{1-|a|^{2}}{|1-x\overline{a}|^{2}}\Big)^{\frac{n}{2}-1} f\circ \varphi_{a}(x)=(1-|a|^{2})^{\frac{n}{2}-1} f(a)P_{\xi}(x),\quad x\in \mathbb{B}_{n}.$$
Thus
$$ f\circ \varphi_{a}(x)= f(a)|1-x\overline{a}|^{n-2}P_{\xi}(x), \quad x\in \mathbb{B}_{n}.$$
Replacing $x$ with $\varphi_{a}(x)$ in the above formula and noticing (\ref{inverse}),  we obtain
$$f(x)=f(a)|1-\varphi_{a}(x)\overline{a}|^{n-2}P_{\xi}\circ \varphi_{a}(x), \quad x\in \mathbb{B}_{n}. $$

Finally, we verify that every positive  harmonic function $f$ defined by (\ref{extremal-function}) satisfies  the equality in  (\ref{Main-sharp})  for all  $x \in \mathbb{B}_{n}$ and $\xi \in   \partial \mathbb{B}_{n}$.  Observing  (\ref{x-y}) and (\ref{1}), straightforward calculations give
\begin{eqnarray*}
  & &  \nabla\log( (1-|x|^{2})^{\frac{n}{2}-1} f(x))
\\
&=& \nabla\log\Big( f(a)(1-|x|^{2})^{\frac{n}{2}-1} \Big(\frac{1-|a|^{2}}{|1-x\overline{a}|}\Big)^{n-2}P_{\xi}\circ \varphi_{a}(x)\Big)
\\
&=&\nabla\log\Big(  \Big(\frac{(1-|a|^{2})(1-|x|^{2})}{|1-x\overline{a}|^{2}}\Big)^{\frac{n}{2}-1}P_{\xi}\circ \varphi_{a}(x)\Big)
\\
&=&\nabla\log( (1-|\varphi_{a}(x)|^{2})^{\frac{n}{2}-1}P_{\xi}\circ \varphi_{a}(x))
\\
&=&\frac{n}{2}\nabla\log \frac{1-|\varphi_{a}(x)|^{2} }{|\varphi_{a}(x)-\xi |^{2}}
\\
&=& (\nabla\log ((1-|\cdot|^{2})^{\frac{n}{2}-1}P_{\xi}(\cdot)) )(\varphi_{a}(x)) \nabla \varphi_{a}(x).
\end{eqnarray*}
Observe  that (\ref{extremal-function-0})  has a equivalent representation
$$|\nabla\log( (1-|x|^{2})^{\frac{n}{2}-1}P_{\xi}(x)) |= \frac{n}{1-|x|^{2}},  \quad (x, \xi)\in \mathbb{B}_{n} \times \partial \mathbb{B}_{n}.$$
Combining this with  $\nabla \varphi_{a}(x)/\|\nabla \varphi_{a}(x)\| \in O(n)$ and  (\ref{gradient-x-y}), we infer that, for all $(x, \xi)\in \mathbb{B}_{n} \times \partial \mathbb{B}_{n}$,
\begin{eqnarray*}
  & & |\nabla\log( (1-|x|^{2})^{\frac{n}{2}-1} f(x))|
\\
&=& |\nabla\log ((1-|\cdot|^{2})^{\frac{n}{2}-1}P_{\xi}(\cdot)) )(\varphi_{a}(x))| \| \nabla \varphi_{a}(x)\|
\\
&=& \frac{n}{1-| \varphi_{a}(x)|^{2}} \| \nabla \varphi_{a}(x)\|
\\
&=& \frac{n}{1-|x|^{2}},
\end{eqnarray*}
which completes the proof.
\end{proof}

  With the help of  Theorem  \ref{Main-sharp-theorem},  we turn  back to prove Theorem  \ref{Main-Ball}.
\begin{proof}[Proof of  Theorem  \ref{Main-Ball}]
By Theorem  \ref{Main-sharp-theorem}, it holds that, for the   harmonic function $f:\mathbb{B}_{n} \rightarrow \mathbb{R}^{+}$,
$$|(|x|^{2}-1)\nabla f(x)+(n-2)xf(x)|\leq n f(x), \quad x\in \mathbb{B}_{n},$$
which has a equivalent  representation
$$|\nabla\log( (1-|x|^{2})^{\frac{n}{2}-1} f(x)) |\leq   \frac{n}{1-|x|^{2}}, \quad x\in \mathbb{B}_{n}.$$
Hence, by Cauchy-Schwarz inequality,
$$|d (\log ((1-|x|^{2})^{\frac{n}{2}-1} f(x)) ) |\leq |\nabla\log( (1-|x|^{2})^{\frac{n}{2}-1} f(x)) | |dx|\leq   \frac{n|dx|}{1-|x|^{2}}.$$
Recall that  the hyperbolic metric on $\mathbb{B}_n$ is given by
$$ d_{\mathbb{B}_{n}}(x,y)=2\tanh^{-1}(|\varphi_{y}(x)|)=\log \frac{1+|\varphi_{y}(x)|}{1-|\varphi_{y}(x)|},$$
and its responding  element of arclength is
$$ds=\frac{2|dx|}{1-|x|^{2}}.$$
Integrating both sides of the above inequality along geodesics for the  hyperbolic metric from $x$ to $y$, we have
$$| \log \frac{(1-|x|^{2})^{\frac{n}{2}-1}f(x)}{(1-|y|^{2})^{\frac{n}{2}-1}f(y)} | \leq \frac{n}{2}d_{\mathbb{B}_{n}}(x,y),$$
which implies that
$$d_{\mathbb{R}^{+}}(f(x),f(y))=| \log \frac{f(x)}{f(y)} | \leq \frac{n}{2}d_{\mathbb{B}_{n}}(x,y)+(\frac{n}{2}-1)| \log \frac{1-|y|^{2}}{1-|x|^{2}} |. $$
Combining this with the inequality
\begin{eqnarray}  \label{Fact}
\frac{1+|\varphi_{y}(x)|}{1-|\varphi_{y}(x)|}\geq  \frac{1-|y|^{2}}{1-|x|^{2}}, \quad x, y \in \mathbb{B}_{n},  \end{eqnarray}
we get
$$d_{\mathbb{R}^{+}}(f(x),f(y))\leq(n-1)d_{\mathbb{B}_{n}}(x,y).$$

To see (\ref{Fact}), using the identity (\ref{1}), we first obtain
$$\frac{  1+|\varphi_{y}(x)| }{1-|\varphi_{y}(x)| } = \frac{(1+|\varphi_{y}(x)|)^{2}}{1-|\varphi_{y}(x)|^{2}}=\frac{(1+|\varphi_{y}(x)|)^{2}|1-x\overline{y}|^{2}}{(1-|x|^{2})(1-|y|^{2})},$$
then the question reduces into    proving that, for $x, y \in \mathbb{B}_{n}$ with  $|x|>|y|$,
\begin{eqnarray} \label{Main-Fact}|1-x\overline{y}|+|x-y|\geq1-|y|^{2}.\end{eqnarray}
If  $|x-y|\geq1-|y|^{2}$, it is a trivial assertion. Otherwise, for $|x-y|<1-|y|^{2}$,
\begin{eqnarray*}
(\ref{Main-Fact})&\Leftrightarrow& |1-x\overline{y}|^{2}\geq(1-|y|^{2}-|x-y|)^{2}=(1-|y|^{2})^{2}+|x-y|^{2}-2|x-y|(1-|y|^{2})
\\
&\Leftrightarrow&(1-|x|^{2})(1-|y|^{2})\geq (1-|y|^{2})^{2}-2|x-y|(1-|y|^{2})
\\
&\Leftrightarrow&  1-|x|^{2} \geq 1-|y|^{2}-2|x-y|
\\
&\Leftrightarrow&   |x|^{2}-|y|^{2} \leq2|x-y|.
\end{eqnarray*}
Under the assumption  that $x, y \in \mathbb{B}_{n}$ with  $|x|>|y|$, it holds naturally that
$$ |x|^{2}-|y|^{2}\leq 2(|x|-|y|)\leq2|x-y|.$$
Now the proof is complete.
\end{proof}

\begin{remark}
In the proof of Theorem  \ref{Main-Ball}, we  give a direct and basic proof of inequality  (\ref{Fact}). As pointed out by an anonymous reader,     (\ref{Fact}) is a consequence of the   known  inequality
 \begin{eqnarray}\label{tri}
\frac{|\rho(x,z)-\rho(z,y)|}{1-\rho(x,z)\rho(z,y)} \leq  \rho(x,y) \leq  \frac{\rho(x,z)+\rho(z,y)}{1+\rho(x,z)\rho(z,y)}, \quad x,y,z\in \mathbb{B}_{n}, \end{eqnarray}
where $\rho(x,y)=|\varphi_{y}(x)|$ is  the so-called pseudo-hyperbolic metric on $\mathbb{B}_{n}$.

To see this, we need only to consider    $x,y\in \mathbb{B}_{n}$ with $|x|\geq |y|$. In this case, it follows that
 \begin{eqnarray}\label{tri-1}\frac{1-|y|^{2}}{1-|x|^{2}}\leq \frac{(1+|x|)(1-|y|)}{(1-|x|)(1+|y|)}
=\frac{1+\frac{|x|-|y|}{1-|x||y|}}{1-\frac{|x|-|y|}{1-|x||y|}}.\end{eqnarray}
According to (\ref{tri}), it holds that
\begin{eqnarray}\label{tri-2}\frac{|x|-|y|}{1-|x||y|} =\frac{|\rho(x,0)-\rho(0,y)|}{1-\rho(x,0)\rho(0,y)} \leq \rho(x,y). \end{eqnarray}
Note that $(1+t)(1-t)^{-1}$ is a increasing function for $t\in (0,1)$,  then (\ref{tri-1}) and (\ref{tri-2}) gives the desired inequality (\ref{Fact}).
\end{remark}

\begin{proof}[Proof of  Theorem  \ref{Theorem-Liu-vector}]
Let $l\in \partial \mathbb{B}_{m}$. For the harmonic $f:\mathbb{B}_{n}\rightarrow \mathbb{B}_{m}$, consider the scalar harmonic function $g=\langle f, l \rangle$, where $\langle \cdot, \cdot\rangle$ is real inner product in $\mathbb{R}^{m}$.  Now the scalar harmonic function $g:\mathbb{B}_{n}\rightarrow (-1,1)$ satisfies the condition of (\ref{Liu}). Hence,
$$|\nabla \langle f(x), l \rangle| = |(\nabla  f(x))^{T}\cdot l  |  \leq \frac{|\mathbb{B}_{n-1}|}{|\mathbb{B}_{n}|}\frac{2}{1-|x|^{2}},\quad x\in \mathbb{B}_{n},$$
where $\cdot$ denotes the matrix product of $(\nabla  f(x))^{T}\in \mathbb{R}^{n\times m}$ with $l\in  \mathbb{R}^{m\times1}$.
Due to the arbitrariness of $l\in \partial \mathbb{B}_{m}$, we obtain
$$\|\nabla  f(x)\|\leq \frac{|\mathbb{B}_{n-1}|}{|\mathbb{B}_{n}|}\frac{2}{1-|x|^{2}},\quad x\in \mathbb{B}_{n},$$
as desired.
\end{proof}

\section{Proof of   Theorems \ref{Theorem-Zhang-improved} and \ref{Wang}}

 Before proving the theorem, we recall the concepts of monogenic functions   in Clifford analysis and octonionic analysis and show that they are subclasses of harmonic functions.

First, we give the definition of  monogenic functions  in Clifford analysis \cite{Gurlebeck}.
\begin{definition}\label{monogenic-Clifford}
Let $\Omega \subset \mathbb{R}^{n+1}$ and   $f:\Omega \rightarrow \mathbb{R}_{0,n}$   be a Clifford algebra valued $C^{1}$ function.  The function $f=\sum_{A}  e_{A}   f_{A}$ is called  (left)  monogenic in $\Omega $ if
$$ Df(x):=\sum _{i=0}^{n}e_{i} \frac{\partial f}{\partial x_{i}}(x)
= \sum _{i=0}^{n} \sum_{A} e_{i}e_{A} \frac{\partial f_{A}}{\partial x_{i}}(x)=0, \quad x\in \Omega. $$
And  the function $f$ is called  (left)  anti-monogenic in $\Omega $ if
$$ \overline{D}f(x):=\sum _{i=0}^{n}\overline{e_{i}} \frac{\partial f}{\partial x_{i}}(x)= \sum _{i=0}^{n} \sum_{A}\overline{ e_{i}}e_{A} \frac{\partial f_{A}}{\partial x_{i}}(x)=0, \quad x\in \Omega. $$
\end{definition}
Due to the non-commutation of Clifford algebra,   the right   monogenic  functions could be defined similarly.
Note that all monogenic functions on $\mathbb{B}_{n+1}$ is real analytic. For the Clifford algebra valued $C^{2}$  functions $f$, by the association of Clifford algebra,  it holds that
\begin{eqnarray}  \label{monogenic-harmonic}
D\overline{D}f(x)=\overline{D} Df(x) =\Delta _{n+1} f(x),
 \end{eqnarray}
where $\Delta _{n+1} $ is Laplace operator in $\mathbb{R}^{n+1}$.

In fact, Definition \ref{monogenic-Clifford} can also be built  in octonionic analysis.
\begin{definition}\label{monogenic-octonion}
Denote by $\mathbb{O}$ the non-commutative and non-associative   algebra with  canonical  vector basis $\{e_{0}=1, e_1,e_2,\ldots,e_7\}$. Let $\Omega \subset \mathbb{O}$ and   $f:\Omega \rightarrow \mathbb{O}$   be a octonionic valued $C^{1}$ function.  The function $f=\sum_{i=0}^{7}  e_{i} f_{i}$ is called  (left)  monogenic in $\Omega $ if
$$ \mathcal{D}f(x):=\sum _{i=0}^{7} e_{i} \frac{\partial f}{\partial x_{i}}(x)
= \sum _{i=0}^{7} \sum_{j=0}^{7} e_{i}e_{j} \frac{\partial f_{j}}{\partial x_{i}}(x)=0, \quad x\in \Omega. $$
And  the function $f$ is called  (left)  anti-monogenic in $\Omega $ if
$$ \overline{\mathcal{D}}f(x):=\sum _{i=0}^{7}\overline{ e_{i} }\frac{\partial f}{\partial x_{i}}(x)
= \sum _{i=0}^{7} \sum_{j=0}^{7}\overline{ e_{i}}e_{j} \frac{\partial f_{j}}{\partial x_{i}}(x)=0, \quad x\in \Omega. $$
\end{definition}

Even though the algebra of octonions is  non-associative,  (\ref{monogenic-harmonic}) still holds in the octonionic setting. Indeed,   the Artin theorem shows  that the subalgebra generated by two elements ($\mathcal{D}$ and $f$) in octonions is associative, which implies
\begin{eqnarray}  \label{monogenic-harmonic-o}
\Delta   f(x) = (\overline{\mathcal{D}}\mathcal{D}) f(x)=\overline{\mathcal{D}}( \mathcal{D}f(x)),
 \end{eqnarray}
where $\Delta $ is Laplace operator in $\mathbb{R}^{8}$.

Hence,  monogenic functions in Clifford analysis and  octonionic analysis belong to   harmonic functions from (\ref{monogenic-harmonic}) and (\ref{monogenic-harmonic-o}).

Since the proof of  Theorem   \ref{Wang} is completely  similar to  Theorem   \ref{Theorem-Zhang-improved},  we only show Theorem   \ref{Theorem-Zhang-improved} in this section.
\begin{proof}[Proof of  Theorem  \ref{Theorem-Zhang-improved}]
Let $f$ be as described in  Theorem  \ref{Theorem-Zhang-improved}. First, if $a=0$ (that is $f(0)=0$), then \cite[Theorem 3.1]{Zhang16} gives that
\begin{eqnarray}  \label{Zhang-0}
|f(x)| \leq \frac{1}{ \sqrt[n+1]{2}-1} |x|, \quad x \in \mathbb{B}_{n+1}.\end{eqnarray}
Now  (\ref{Zhang-improved}) at $x=0$ is obtained. Otherwise,  as in the   prove of  Theorem  \ref{Main-Ball}, consider the Clifford algebra valued  harmonic function
 $$g_{1}(x)=\Big(\frac{1-|a|^{2}}{|1-x\overline{a}|^{2}}\Big)^{\frac{n+1}{2}-1} f\circ \varphi_{a}(x), \quad x\in \mathbb{B}_{n+1}.$$
In view of the estimate  (\ref{gradient}), set
$$g(x)= \Big(\frac{1-|a|}{1+|a|}\Big)^{\frac{n-1}{2}}g_{1}(x)=\Big(\frac{1-|a|}{|1-x\overline{a}|}\Big)^{ n-1 } f\circ \varphi_{a}(x)$$
with $|g(x)|<1$ for $x\in \mathbb{B}_{n+1}$.
 Applying   the inequality (\ref{Zhang-0}) to the   harmonic function $g(x)$,
we obtain
$$  \Big(\frac{1-|a|}{|1-x\overline{a}|}\Big)^{ n-1 } |f\circ \varphi_{a}(x)|\leq \frac{1}{ \sqrt[n+1]{2}-1} |x|, \quad x\in \mathbb{B}_{n+1}.$$
Let $y=\varphi_{a}(x)$. From the identity (\ref{x-y}), we have
$$   |f (y)|\leq \frac{1}{ \sqrt[n+1]{2}-1} \Big(\frac{1+|a|}{|1-y\overline{a}|}\Big)^{ n-1 } |\varphi_{a}^{-1}(y)|, \quad y\in \mathbb{B}_{n+1}.$$
Thus the fact $\varphi_{a}=\varphi_{a}^{-1}$ in (\ref{inverse})  gives the desired inequality
$$|f(x)| \leq \frac{(1+|a|)^{n}}{ \sqrt[n+1]{2}-1} \frac{|x-a|}{|1-\overline{a}x|^{n+1}}, \quad x \in \mathbb{B}_{n+1}.$$
The proof is completed.
\end{proof}
\textbf{Declarations}
\\
\textbf{Conflict of interest}
There is no financial or non-financial interests that are directly or indirectly related to the work submitted for publication.
\\
\textbf{Funding}
 This work was supported by  the Anhui Provincial Natural Science Foundation (No. 2308085MA04),  the   National Natural Science Foundation of China  (Nos. 11801125, 12301097) and the Fundamental Research Funds for the Central Universities (Nos. JZ2023HGQA0117, JZ2023HGTA0169). \\
\textbf{Data availability}
Data sharing is not applicable to this article as no datasets were generated  during the current study.

\bibliographystyle{amsplain}

\begin{thebibliography}{99}

\bibitem{Ahlfors} Lars V.  Ahlfors,  \textit{M\"{o}bius transformations in several dimensions}, Ordway Professorship Lectures in Mathematics.   Minneapolis, Minn., 1981.

\bibitem{Axler} Sheldon Axler, Paul Bourdon, Wade Ramey, \textit{Harmonic function theory. Second edition}, Graduate Texts in Mathematics, 137. Springer-Verlag, New York, 2001.

\bibitem{Chen} Huaihui Chen, \textit{The Schwarz-Pick lemma and Julia lemma for real planar harmonic mappings}, Sci. China Math. 56 (2013), no. 11, 2327-2334.


\bibitem{Chen-Hamada}Shaolin Chen, Hidetaka Hamada, \textit{Some sharp Schwarz-Pick type estimates and their applications of harmonic and pluriharmonic functions}, J. Funct. Anal. 282 (2022), no. 1, Paper No. 109254, 42 pp.
\bibitem{Chen-Hamada-Ponnusamy} Shaolin Chen, Hidetaka Hamada, Saminathan Ponnusamy, Ramakrishnan Vijayakumar, \textit{Schwarz type lemmas and their applications in Banach spaces}, To appear in Journal d'Analyse Math\'{e}matique.

%\bibitem{Dai}Shaoyu  Dai, Huaihui Chen, \textit{A Schwarz lemma for harmonic functions in the real unit ball}, Acta Math. Sci. Ser. B (Engl. Ed.) 39 (2019), no. 5, 1339-1344.
\bibitem{Gu}Longfei Gu, \textit{Schwarz-type lemmas associated to a Helmholtz equation}, Adv. Appl. Clifford Algebr. 30 (2020), no. 1, Paper No. 14, 10 pp.
\bibitem{Gurlebeck} Klaus G\"{u}rlebeck, Klaus   Habetha,  Wolfgang Spr\"{o}{\ss}ig, \textit{Holomorphic functions in the plane and $n$-dimensional
space}, Birkh\"{a}user Verlag, Basel, 2008.

\bibitem{Hua}Loo Keng Hua,
\textit{Starting with the unit circle},
Background to higher analysis. Translated from the Chinese by Kuniko Weltin. Springer-Verlag, New York-Berlin, 1981.

%\bibitem{Kalaj} David  Kalaj, \textit{Schwarz lemma for holomorphic mappings in the unit ball},Glasg. Math. J. 60 (2018), no. 1, 219-224.
\bibitem{Kalaj}David Kalaj,  \textit{Heinz-Schwarz inequalities for harmonic mappings in the unit ball}, Ann. Acad. Sci. Fenn. Math. 41 (2016), no. 1, 457-464.
\bibitem{Kalaj-Vuorinen}David Kalaj,  Matti Vuorinen, \textit{On harmonic functions and the Schwarz lemma}, Proc. Amer. Math. Soc. 140 (2012), no. 1, 161-165.
\bibitem{Khalfallah}Adel  Khalfallah, Miodrag Mateljevi\'{c},  Bojana Purti\'{c},
\textit{Schwarz-Pick lemma for harmonic and hyperbolic harmonic functions},
Results Math. 77 (2022), no. 4, Paper No. 167.
\bibitem{Khavinson} Dmitry   Khavinson, \textit{An extremal problem for harmonic functions in the ball},
Canad. Math. Bull. 35 (1992), no. 2, 218-220.

\bibitem{Liu}Congwen Liu, \textit{A proof of the Khavinson conjecture}, Math. Ann. 380 (2021), no. 1-2, 719-732.
\bibitem{Liu2}Congwen Liu, \textit{Schwarz-Pick lemma for harmonic functions},  Int. Math. Res. Not. IMRN 2022, no. 19, 15092-15110.
\bibitem{Markovic} Marijan  Markovi\'{c},
 \textit{On harmonic functions and the hyperbolic metric}, Indag. Math. (N.S.) 26 (2015), no. 1, 19-23.
\bibitem{Melentijevic-M} Miodrag Mateljevi\'{c},  \textit{Schwarz lemma and Kobayashi metrics for harmonic and holomorphic functions}, J. Math. Anal. Appl. 464 (2018), no. 1, 78-100.

\bibitem{Melentijevic-P} Petar Melentijevi\'{c},
\textit{Invariant gradient in refinements of Schwarz and Harnack inequalities},
Ann. Acad. Sci. Fenn. Math. 43 (2018), no. 1, 391-399.

\bibitem{Pavlovic}Miroslav Pavlovi\'{c}, \textit{A Schwarz lemma for the modulus of a vector-valued analytic function}, Proc. Amer. Math. Soc. 139 (2011), no. 3, 969-973.

\bibitem{Stoll} Manfred  Stoll, \textit{Harmonic and subharmonic function theory on the hyperbolic ball}, London Mathematical Society Lecture Note Series, 431. Cambridge University Press, Cambridge, 2016.


 \bibitem{Wang-Bian-Liu} Haiyan Wang,  Xiaoli Bian,  Hua Liu, \textit{M\"{o}bius transformation and a version of Schwarz lemma in octonionic analysis}, Math. Methods Appl. Sci. 44 (2021), no. 1, 27-42.

      \bibitem{Wang-Sun-Bian}Haiyan Wang, Ningxin Sun,  Xiaoli Bian, \textit{A version of Schwarz lemma associated to the k-Cauchy-Fueter operator}, Adv. Appl. Clifford Algebr. 31 (2021), no. 4, Paper No. 64, 14 pp.

 \bibitem{Xu}    Zhenghua Xu,  \textit{Schwarz lemma for pluriharmonic functions}, Indag. Math. (N.S.) 27 (2016), no. 4, 923-929.
 \bibitem{Xu2}    Zhenghua Xu,  \textit{A Schwarz-Pick lemma for the norms of holomorphic mappings in Banach spaces}, Complex Var. Elliptic Equ. 63 (2018), no. 10, 1459-1467.

  \bibitem{Yang-Qian} Yan Yang,  Tao  Qian,\textit{Schwarz lemma in Euclidean spaces}, Complex Var. Elliptic Equ, 51 (2006), no. 7, 653-659.

  \bibitem{Zhang14}Zhongxiang Zhang, \textit{The Schwarz lemma in Clifford analysis}, Proc. Amer. Math. Soc. 142 (2014), no. 4, 1237-1248.
 \bibitem{Zhang16} Zhongxiang Zhang, \textit{The Schwarz lemma for functions with values in $C(V_{n,0})$}, J. Math. Anal. Appl. 443 (2016), no. 2, 1130-1141.
 
\end{thebibliography}

\vskip 10mm

\end{document}